\documentclass[11pt]{article}

\usepackage{amsmath}
\usepackage{amssymb}
\usepackage{amsfonts}
\usepackage{amscd}
\usepackage{mathtools}
\usepackage{diffcoeff}
\usepackage{amsthm}
\usepackage{latexsym}
\usepackage{bm}
\usepackage{graphicx}
\usepackage{caption}
\usepackage{hyperref}
\usepackage[all]{xy}
\usepackage[top=20mm,bottom=30mm,left=25mm,right=25mm]{geometry}

\theoremstyle{definition}
\newtheorem{Def}{Definition}[section]
\newtheorem{Lem}[Def]{Lemma}
\newtheorem{Fact}[Def]{Fact}
\newtheorem{Prop}[Def]{Proposition}
\newtheorem{Thm}[Def]{Theorem}

\newtheorem{Ex}[Def]{Example}

\newtheorem{Cor}[Def]{Corollary}

\AtBeginDocument{%
}

\title{A Geometric Interpretation of Generalized Hurwitz--Radon Numbers Defined by Kannaka--Tojo}
\author{Muneto Miyaji}
\date{}

\begin{document}
\maketitle

\begin{abstract}
The Hurwitz--Radon number originates in the composition problem of quadratic forms and is related to the maximum number of pointwise linearly independent vector fields on spheres. 
Kannaka--Tojo [arXiv:2602.04544] reformulated the Hurwitz--Radon number in the setting of a real reductive Lie algebra $\mathfrak g$ and its faithful representation $\iota$, and introduced two natural numbers $\rho^{(1)}(\mathfrak g,\iota)$ and $\rho^{(2)}(\mathfrak g,\iota)$. For classical Lie algebras and their standard representations, these two numbers coincide except for a few cases.

In this paper, fixing a Lie group $G$ and a subspace $\mathfrak{s} $ of $ \mathfrak g = \operatorname{Lie}(G)$ , we define natural numbers $\rho_{G,\mathfrak{s}}(M,\sigma)$ and $\rho^{\pm}_{G,\mathfrak{s}}(M,\sigma,\nabla)$ for a $G$-manifold $(M,\sigma)$ equipped with an affine connection $\nabla$. These are defined in terms of fundamental vector fields on $M$. In a special case, we show that $\rho_{G,\mathfrak{s}}(M,\sigma)$ coincides with $\rho^{(2)}(\mathfrak g,\iota)$, and that $\rho^{-}_{G,\mathfrak{s}}(M,\sigma,\nabla)$ coincides with $\rho^{(1)}(\mathfrak g,\iota)$. Furthermore, we show that $\rho^{+}_{G,\mathfrak{s}}(M,\sigma,\nabla)$ is related to Clifford structures on $M$.

\end{abstract}

\tableofcontents

\section{Introduction}


The Hurwitz--Radon number $\rho(N)$, associated with each positive integer $N$, is a classical invariant arising from the composition problem of quadratic forms, originating in the work of Hurwitz \cite{Hurwitz22} and Radon \cite{Radon22}.

When a positive integer $N$ is expressed as
    $$
    N = 2^{4a+b}(2c+1)\quad (a, b, c \in \mathbb{N} ,0\leq b \leq 3),
    $$
    the Hurwitz--Radon number $\rho(N)$ is defined by
    $$
    \rho(N) = 8a + 2^b.
    $$

Adams \cite{Adams62} gave a geometric characterization of the Hurwitz--Radon number by showing that $\rho(N)-1$ coincides with the maximum number of pointwise linearly independent vector fields on the sphere $S^{N-1}$.
Furthermore, the Hurwitz--Radon number is known to appear in various contexts, such as the invertibility problem for linear combinations of  real matrices \cite{AuYik71} and the study of Clifford--Klein forms by Kobayashi--Yoshino \cite{KobayashiYoshino05}.


Although the Hurwitz--Radon number is defined algebraically for positive integers, it has recently been generalized by Kannaka--Tojo \cite{KannakaTojo26} in the framework of real reductive Lie algebras and their representations.

Let $(\mathfrak{g},\iota)$ be a pair of a real reductive Lie algebra $\mathfrak g$ and a faithful representation $\iota:\mathfrak g\to\mathfrak{gl}(N,\mathbb C).$ Assume that $\iota(\mathfrak g)$ is self-adjoint. Let $\theta$ be the associated involution of $\mathfrak g$, and let $\mathfrak g=\mathfrak k\oplus\mathfrak p$ be the associated Cartan decomposition(see Section 2 for details).

\begin{Def}[Kannaka--Tojo  \cite{KannakaTojo26}]
    In the setting above, we define two integers $\rho^{(i)}(\mathfrak{g}, \iota )$ $(i = 1, 2)$ as the largest $n \in \mathbb{N}$ for which there exists an $\mathbb{R}$-linear map $f:\mathbb{R}^n\to \mathfrak{p}$ such that
    \begin{description}
\item[$\bullet$]$\rho^{(1)}(\mathfrak{g}, \iota)$ : $\iota(f(v))^2 = \| v \|^2I_N$ for any $v \in \mathbb{R}^n$,
\item[$\bullet$]$\rho^{(2)}(\mathfrak{g}, \iota)$ : $\iota(f(v))$ is invertible for any non-zero $v \in \mathbb{R}^n$.
 \end{description}
 Here, $\| v \|$ denotes the standard norm of $v \in \mathbb{R}^n$.
\end{Def}

Clearly,
$$
\rho^{(1)}(\mathfrak{g}, \iota) \leq \rho^{(2)}(\mathfrak{g}, \iota)
$$
holds.

For example, for the standard representation
$\iota:\mathfrak{so}(N,N)\to\mathfrak{gl}(2N,\mathbb{C})$,
one has
\[
\rho^{(1)}(\mathfrak{so}(N,N),\iota)
=
\rho^{(2)}(\mathfrak{so}(N,N),\iota)
=
\rho(N).
\]
Thus, $\rho^{(1)}(\mathfrak{g}, \iota)$ and $\rho^{(2)}(\mathfrak{g}, \iota)$ may be regarded as generalizations of the Hurwitz--Radon number.

For classical Lie algebras $\mathfrak{g}$ and their standard representations $\iota$, the values of $\rho^{(1)}(\mathfrak g,\iota)$ and $\rho^{(2)}(\mathfrak g,\iota)$ were completely determined by Kannaka--Tojo, and if $\mathfrak{g} \neq \mathfrak{sl}(2N+1, \mathbb{D})$ $(\mathbb{D} = \mathbb{R}, \mathbb{C}, \mathbb{H}, N \geq 1)$, then
$$
\rho^{(1)}(\mathfrak{g},\iota)=\rho^{(2)}(\mathfrak{g},\iota)
$$
holds.

Moreover, for certain symmetric spaces $G/H$ and representations
$\iota:\mathfrak{g}\to\mathfrak{gl}(N,\mathbb{C})$,
Kannaka--Tojo showed that $\rho^{(1)}(\mathfrak g,\iota)$ determines the largest $n$
for which $\mathrm{Spin}(n,1)$ can act properly on $G/H$.
Details are given in Section 2.



The Hurwitz--Radon number is geometrically characterized by vector fields on spheres. Furthermore, Kannaka--Tojo introduced generalized Hurwitz--Radon numbers in the representation-theoretic setting, which naturally raises the following question:

\medskip
\noindent
\textbf{Problem.}
Can $\rho^{(1)}(\mathfrak g,\iota)$ and $\rho^{(2)}(\mathfrak g,\iota)$
be reformulated in a more geometric framework?


In this paper, we reformulate the generalized Hurwitz--Radon numbers introduced by Kannaka--Tojo
in the setting of $G$-manifolds.

Throughout this paper, we denote the Lie algebra of a Lie group $G$ by
$\mathfrak g$.

The notion of a $(G,\mathfrak{s})$-fundamental vector field used below
is defined in Definition~\ref{d31}.

\begin{Def}
Let $G$ be a Lie group, $M$ a $G$-manifold with action $\sigma$, and let
$\mathfrak{s}\subset\mathfrak g$ be a vector subspace.
We define
\[
\rho_{G,\mathfrak{s}}(M,\sigma)
\coloneqq
\max\left\{
n\in\mathbb N \,\middle|\,
\begin{array}{l}
\text{there exist } n \textrm{ } (G, \mathfrak{s})\text{-fundamental vector fields on }M\\
\text{which are linearly independent at every point}
\end{array}
\right\}.
\]

Let $\nabla$ be an affine connection on $M$, and define
\[
\nabla^{M}\mathfrak{s}
\coloneqq
\{\nabla X_A \mid A\in \mathfrak{s}\}
\subset \operatorname{End}_{C^\infty(M)}(\mathfrak X(M)),
\]
where $X_A$ denotes the fundamental vector field on $M$ corresponding to $A\in\mathfrak g$.
For each $\epsilon\in\{\pm1\}$, we define integer
$\rho^\epsilon_{G,\mathfrak{s}}(M,\sigma,\nabla)$
as the largest $n\in\mathbb N$ for which there exists a unital $\mathbb R$-algebra homomorphism
$f:Cl_n^\epsilon\to \operatorname{End}_{C^\infty(M)}(\mathfrak X(M))$ 
such that
\begin{itemize}
\item $f(e_i)\in\nabla^M \mathfrak{s}$ for any $i \in \{1, ..., n\}$.
\end{itemize}
Here, the notation $+$ (resp. $-$) corresponds to $\epsilon=1$ (resp. $\epsilon=-1$), and
$Cl_n^+=Cl_n$, $Cl_n^-=Cl_{0,n}$.
\end{Def}

Our first result shows that, in a special setting, these invariants coincide with the generalized Hurwitz--Radon numbers of Kannaka--Tojo.

\begin{Thm}\label{t13}
Let $G$ be a real reductive linear Lie group, let
$\varphi:G\to GL(N,\mathbb C)$ be a Lie group homomorphism, and set
$\iota\coloneqq(d\varphi)_e$.
Assume that $\iota$ is faithful and that $\iota(\mathfrak g)$ is
self-adjoint. Let
$\mathfrak g=\mathfrak k\oplus\mathfrak p$
be the associated Cartan decomposition. Then
\[
\rho_{G,\mathfrak p}(\mathbb C^N\backslash\{0\},\varphi)
=
\rho^{(2)}(\mathfrak g,\iota).
\]

Moreover, let $\nabla$ be the standard flat affine connection on
$\mathbb C^N\backslash\{0\}$. Then
\[
\rho^-_{G,\mathfrak p}(\mathbb C^N\backslash\{0\},\varphi,\nabla)
=
\rho^{(1)}(\mathfrak g,\iota).
\]
\end{Thm}

As a modest application, we obtain the following.

\begin{Thm}\label{t14}
Let $N\in\mathbb{N}$, let $\mathfrak{so}(N,N)=\mathfrak{k}+\mathfrak{p}$ be a Cartan decomposition, and let
$\varphi:SO(N,N)\to GL(2N,\mathbb{R})$
be the standard representation. Then the following holds:
\begin{align*}
\rho(N)
&=
\max\left\{
n\in\mathbb N \,\middle|\,
\begin{array}{l}
\text{there exist } n \textrm{ }(O(N), \mathfrak{o}(N)) \text{-fundamental vector fields on }\mathbb{R}^{N}\backslash\{0\}\\
\text{which are linearly independent at every point}
\end{array}
\right\}+1\\
&=
\max\left\{
n\in\mathbb N \,\middle|\,
\begin{array}{l}
\text{there exist } n \textrm{ }(SO(N, N), \mathfrak{p}) \text{-fundamental vector fields on }\mathbb{R}^{2N}\backslash\{0\}\\
\text{which are linearly independent at every point}
\end{array}
\right\}.
\end{align*}
\end{Thm}

Our final result shows that, in a special setting, $\rho^+_{G,\mathfrak{s}}$ is closely related to Clifford structures.

Let $M$ be a $G$-manifold equipped with a $G$-invariant Riemannian
metric $g$, let $\sigma$ be the action of $G$ on $M$, and let $\nabla$
be the Levi--Civita connection on $M$.
\begin{Thm}\label{t16}
    For $n\in\mathbb N$, the following are equivalent.
    \begin{description}
\item[(1)]$n \leq \rho^+_{G, \mathfrak{s}}(M, \sigma, \nabla)$.
\item[(2)]There exists a rank $n$ Clifford structure $(E, h, \varphi)$ satisfying the following.
\begin{description}
\item[$\bullet$]$E$ is trivial.
\item[$\bullet$]There exist sections $s_1, ..., s_n \in \Gamma(E)$ forming an $h$-orthonormal frame such that for each $i\in \{1, ..., n\}$, $\varphi\circ s_i \in \widetilde{\nabla}^M\mathfrak{s}$.
 \end{description}
 \end{description}
 Here, \(\widetilde{\nabla}:\mathfrak X(M)\to \operatorname{End}_{C^\infty(M)}(\mathfrak X(M))\cong \Gamma(\operatorname{End}(TM))\) denotes the map induced by \(\nabla\).
\end{Thm}

\section{Preliminaries}

\subsection{Hurwitz--Radon numbers associated with Lie algebras by Kannaka--Tojo}

In this subsection, we review the generalized Hurwitz--Radon numbers introduced by Kannaka--Tojo \cite{KannakaTojo26} and summarize their basic definitions and main results.

Let $(\mathfrak{g}, \iota)$ be a pair of a real reductive Lie algebra $\mathfrak{g}$ and a faithful representation
$
\iota:\mathfrak{g} \to \mathfrak{gl}(N, \mathbb{C}).
$
We assume that the image $\iota(\mathfrak{g})$ is self-adjoint, namely, $\iota(\mathfrak{g})$ is closed under taking adjoint operators with respect to some positive definite hermitian form $\langle, \rangle$ on $\mathbb{C}^N$. Then there exists an involutive automorphism $\theta$ of $\mathfrak{g}$ such that
$$
\langle \iota(X)v, w\rangle = - \langle v, \iota (\theta(X))w\rangle \quad (X \in \mathfrak{g}, v, w \in \mathbb{C}^N).
$$
We denote by $\mathfrak{g} = \mathfrak{k} +\mathfrak{p}$ the  eigenspace decomposition of $\mathfrak{g}$ with respect to $\theta$ namely, $\mathfrak{k}$ (resp. $\mathfrak{p}$) is the eigenspace of $\theta$ with eigenvalue $1$ (resp. $-1$). If $\mathfrak{g}$ is semisimple, then any faithful representation $\iota$ of $\mathfrak{g}$ satisfies our assumption, and $\theta$ is a Cartan involution of $\mathfrak{g}$.

\begin{Def}[Kannaka--Tojo  \cite{KannakaTojo26}]
    In the setting above, we define two integers $\rho^{(i)}(\mathfrak{g}, \iota )$ $(i = 1, 2)$ as the largest $n \in \mathbb{N}$ for which there exists an $\mathbb{R}$-linear map $f:\mathbb{R}^n\to \mathfrak{p}$ such that
    \begin{description}
\item[$\bullet$]$\rho^{(1)}(\mathfrak{g}, \iota)$ : $\iota(f(v))^2 = \| v \|^2I_N$ for any $v \in \mathbb{R}^n$,
\item[$\bullet$]$\rho^{(2)}(\mathfrak{g}, \iota)$ : $\iota(f(v))$ is invertible for any non-zero $v \in \mathbb{R}^n$.
 \end{description}
 Here, $\| v \|$ denotes the standard norm of $v \in \mathbb{R}^n$.
\end{Def}

Clearly,
$$
\rho^{(1)}(\mathfrak{g}, \iota) \leq \rho^{(2)}(\mathfrak{g}, \iota)
$$
holds.

For example, for the standard representation
$\iota:\mathfrak{so}(N,N)\to\mathfrak{gl}(2N,\mathbb{C})$,
one has
\[
\rho^{(1)}(\mathfrak{so}(N,N),\iota)
=
\rho^{(2)}(\mathfrak{so}(N,N),\iota)
=
\rho(N).
\]

For classical Lie algebras and their standard representations, the values of $\rho^{(1)}(\mathfrak g,\iota)$ and $\rho^{(2)}(\mathfrak g,\iota)$ were completely determined by Kannaka--Tojo.
\begin{Fact}\label{f22}
[Kannaka--Tojo {\cite[Theorem B]{KannakaTojo26}}]\label{fact:KT-ThmB}
Let $\mathfrak{g}$ be a classical Lie algebra and $\iota$ its standard representation. If $\mathfrak{g} \neq \mathfrak{sl}(2N+1, \mathbb{D})$ $(\mathbb{D} = \mathbb{R}, \mathbb{C}, \mathbb{H}, N \geq 1)$, then
$$
\rho^{(1)}(\mathfrak{g},\iota)=\rho^{(2)}(\mathfrak{g},\iota),
$$
and the common value is summarized in Table 1. 

Moreover, in the exceptional case $\mathfrak{g} = \mathfrak{sl}(2N+1, \mathbb{D})$, we have
$$
\rho^{(1)}(\mathfrak{sl}(2N+1, \mathbb{D}),\iota) = 0, \quad \rho^{(2)}(\mathfrak{sl}(2N+1, \mathbb{D}),\iota) = 1.
$$
\begin{table}[htbp]
\centering
\small
\caption{Values of $\rho^{(i)}(\mathfrak g,\iota)$ associated with classical pairs}
\label{tab:KT}
\renewcommand{\arraystretch}{1.2}
\begin{tabular}{|c|c||c|c|}
\hline
$\mathfrak g$ & $\rho^{(1)}=\rho^{(2)}$
& $\mathfrak g$ & $\rho^{(1)}=\rho^{(2)}$ \\
\hline
 $\mathfrak{so}(N,N)$ & $\rho(N)$
& $\mathfrak{gl}(N,\mathbb C)$ & $2\,\mathrm{ord}_2(N)+1$ \\

$\mathfrak{gl}(N,\mathbb R)$ & $\rho(N/2)+1$& $\mathfrak{su}(N,N)$ & $2\,\mathrm{ord}_2(N)+2$ \\

$\mathfrak{sp}(N,\mathbb R)$ & $\rho(N/2)+2$& $\mathfrak{sl}(2N,\mathbb R)$ & $\rho(N)+1$ \\

$\mathfrak{sp}(N,\mathbb C)$ & $\rho(N/2)+3$ & $\mathfrak{sl}(2N,\mathbb C)$ & $2\,\mathrm{ord}_2(N)+3$ \\

$\mathfrak{sp}(N,N)$ & $\rho(N/2)+4$& $\mathfrak{sl}(2N,\mathbb H)$ & $\rho(N/2)+5$ \\

$\mathfrak{gl}(N,\mathbb H)$ & $\rho(N/4)+5$& $\mathfrak{sl}(1,\mathbb D)$ & $0$ \\

$\mathfrak{so}^*(2N)$ & $\rho(N/8)+6$& $\mathfrak{su}(p,q;\mathbb D)\ (p\neq q)$ & $0$ \\

$\mathfrak{so}(N,\mathbb C)$ & $\rho(N/16)+7$
& & \\

\hline
\end{tabular}
\end{table}
\end{Fact}

The integer $\rho^{(1)}(\mathfrak g,\iota)$ can also be described in terms of Lie group homomorphisms from $Spin(n,1)$.

\begin{Fact}[Kannaka--Tojo {\cite[Proposition 5.4]{KannakaTojo26}}]\label{f23}
Let $G$ be a classical semisimple Lie subgroup of $GL(N,\mathbb{C})$,
$\tilde{\iota}:G \to GL(N, \mathbb{C})$ the inclusion map, and $\iota \coloneqq (d\tilde{\iota})_e$ the differential map.
Then the following claims are equivalent for any integer $n \geq 2$:
\begin{description}
\item[(1)]$n \leq \rho^{(1)}(\mathfrak{g}, \iota)$;
\item[(2)]There exists a Lie group homomorphism $\varphi:Spin(n, 1) \to G$ such that $\tilde{\iota}(\varphi(-1)) = -I_N$;
\item[(3)]There exists a Lie group homomorphism $\varphi:Spin(n, 1) \to G$ such that the representation $\tilde{\iota}\circ \varphi$ equivalent to a direct sum of
several copies of the spin representation $S$ when $n$ is even, and
to a direct sum of several copies of the semispin representations
$S_1$ and $S_2$ when $n$ is odd.
 \end{description}
\end{Fact}

Moreover, $\rho^{(1)}(\mathfrak g,\iota)$ gives the largest $n$ for which $\mathrm{Spin}(n,1)$ can act properly on $G/H$.
\begin{Fact}[Kannaka--Tojo {\cite[Theorem A, D]{KannakaTojo26}}]
Fix $n \geq 2$. Let $G/H$ be a semisimple symmetric space, and assume that $\mathfrak{g}$ is simple. Also assume that $G$ admits a connected complexification. If $G$ and $H$ are locally isomorphic to one of the pairs in Table~\ref{tab:symmetric-pairs} and $\iota$ is the standard representation of $\mathfrak{g}$, then the following are equivalent:
$\mathrm{Spin}(n,1)$ acts properly on $G/H$, and
$n\le \rho^{(1)}(\mathfrak{g},\iota)$.
\end{Fact}

\begin{table}[htbp]
\centering
\small
\caption{Where $0\leq p \leq N/2$}
\label{tab:symmetric-pairs}
\begin{tabular}{cc} 
   $G$ & $H$ \\ \hline
   $SL(2N, \mathbb{R})$ & $SO(N+1, N-1)$  \\
   $SL(2N, \mathbb{C})$ & $SU(N+1, N-1)$  \\ 
   $SL(2N, \mathbb{H})$ & $Sp(N+1, N-1)$  \\
   $SO^{*}(4N)$ & $U(N+1, N-1)$  \\ 
   $SO(2N, \mathbb{C})$ & $SO(N+1, N-1)$  \\
   $SO(N, N)$ & $SO(p, p+1) \times  SO(N-p, N-p-1)$  \\ 
   $SU(N, N)$ & $S(U(p, p+1) \times  U(N-p, N-p-1))$  \\ 
   $Sp(N, N)$ & $Sp(p, p+1) \times  Sp(N-p, N-p-1)$  \\ 
   $SO^*(4N)$ & $SO^*(4p+2) \times  SO^*(4N-4p-2)$  \\ 
   $SO(2N, \mathbb{C})$ & $SO(2p + 1, \mathbb{C}) \times  SO(2N-2p -1, \mathbb{C})$  \\ 
\end{tabular}
\end{table}

\subsection{Clifford structures on Riemannian manifolds}

In this subsection, we recall the definition of Clifford structures on Riemannian manifolds following Moroianu--Semmelmann \cite{AndreiUwe11}.

Let $(M,g)$ be a Riemannian manifold. For $A\in\operatorname{End}(TM)$, we say that $A$ is skew-symmetric if
\[
g_p(Av,w)=-g_p(v,Aw)
\]
for all $p\in M$ and all $v,w\in T_pM$. We denote by
$\operatorname{End}^{-}(TM)$
the set of all skew-symmetric endomorphisms of $TM$.

Let $(E,h)$ be a Euclidean bundle over $M$. The bundle
\[
Cl(E,h)\coloneqq\bigsqcup_{x\in M} Cl(E_x,h_x)
\]
is called the Clifford bundle. Here, for each $x\in M$, $Cl(E_x,h_x)$ denotes the Clifford algebra of $(E_x,h_x)$.

\begin{Def}
    A rank $r$ Clifford structure on Riemannian manifold $(M, g)$ is an oriented rank $r$ Euclidean bundle over $M$ together with a non-vanishing algebra bundle morphism called Clifford morphism, $\varphi:Cl(E, h) \to \textrm{End}(TM)$ which maps $E$ into $\textrm{End}^-(TM)$.
\end{Def}

\section{Hurwitz--Radon numbers for $G$-manifolds}
In this section, we define $\rho_{G,\mathfrak{s}}(M,\sigma)$ for $G$-manifolds and show that $\rho_{G,\mathfrak{s}}(M,\sigma)$ is an extension of $\rho^{(2)}(\mathfrak g,\iota)$ defined by Kannaka--Tojo. This provides a geometric interpretation of $\rho^{(2)}(\mathfrak g,\iota)$.

\subsection{Definition and properties of $\rho_{G, \mathfrak{s}}(M, \sigma)$}
Let $G$ be a Lie group, $M$ a $G$-manifold, and $\sigma$ its group action. Let also $\mathfrak{s} \subset \mathfrak{g}$ be a vector subspace.

\begin{Def}\label{d31}
For $A\in\mathfrak g$, the vector field $X_A\in\mathfrak X(M)$ defined by
$$
X_A: M \to TM,\quad p \mapsto (d\sigma_p)_e(A),
$$
where $\sigma_p:G\to M,\ g\mapsto \sigma(g,p)$, is called the fundamental vector field corresponding to $A$.

In particular, for $A\in \mathfrak{s}$, we call $X_A$ a $(G,\mathfrak{s})$-fundamental vector field.
\end{Def}

For a $G$-manifold, we define the following integer.
\begin{Def}
    \[
\rho_{G,\mathfrak{s}}(M,\sigma)
\coloneqq
\max\left\{
n\in\mathbb N \,\middle|\,
\begin{array}{l}
\text{there exist } n \textrm{ } (G, \mathfrak{s})\text{-fundamental vector fields on }M\\
\text{which are linearly independent at every point}
\end{array}
\right\}.
\]
\end{Def}

First, we discuss the functoriality of $\rho_{G,\mathfrak{s}}$.
\begin{Def}
We define a category $G\textrm{-}\mathbf{Mfd}$ as follows.

\begin{itemize}
\item An object of $G\textrm{-}\mathbf{Mfd}$ is a pair $(M,\sigma)$, where $M$ is a smooth manifold and
$\sigma:G\times M\to M$
is a smooth $G$-action.

\item A morphism
$\psi:(M,\sigma)\to (N,\tau)$
is a smooth map $\psi:M\to N$ such that:
\begin{enumerate}
\item $\psi$ is surjective;
\item $\psi$ is $G$-equivariant;
\item for each $p\in M$, the restriction
$\psi|_{G\cdot p}:G\cdot p\to N$
is an immersion.
\end{enumerate}
\end{itemize}
\end{Def}

\begin{Prop}\label{p34}
 $\rho_{G,\mathfrak{s}}$ is a functor from $G\textrm{-}\mathbf{Mfd}$ to the partially ordered set $(\mathbb{Z}_{\geq 0},\leq)$.
\end{Prop}
\begin{proof}
  It suffices to show that for any
$(M,\sigma),\ (N,\tau)\in \mathrm{Ob}(G\textrm{-}\mathbf{Mfd})$,
if there exists a morphism
$\psi:(M,\sigma)\to (N,\tau)$,
then
$\rho_{G,\mathfrak{s}}(M,\sigma)\leq \rho_{G,\mathfrak{s}}(N,\tau)$.
Let $\psi:M\to N$ be a morphism from $(M,\sigma)$ to $(N,\tau)$. Put
$n:=\rho_{G,\mathfrak{s}}(M,\sigma)$.
Let $X_1,\dots,X_n$ be pointwise linearly independent fundamental vector fields on $M$ with respect to $\sigma$, corresponding to $A_1,\dots,A_n\in \mathfrak{s}$. Let $Y_1,\dots,Y_n$ be the fundamental vector fields on $N$ with respect to $\tau$, corresponding to the same elements $A_1,\dots,A_n$.
Then, for each $q\in N$ and each $p\in \psi^{-1}(q)$, we have
\begin{align*}
Y_i(q)
&=(d\tau_q)_e(A_i)
=(d(\psi\circ \sigma_p))_e(A_i) \\
&=((d\psi)_p\circ (d\sigma_p)_e)(A_i)
=(d\psi)_p(X_i(p)).
\end{align*}
Now suppose that $t_1,\dots,t_n\in\mathbb R$ satisfy
\[
\sum_{i=1}^n t_i Y_i(q)=0.
\]
Then
\[
\sum_{i=1}^n t_iY_i(q)
=
(d\psi)_p(\sum_{i=1}^n t_iX_i(p))
=
(d\psi)_p(\sum_{i=1}^n t_i(d\sigma_p)_e(A_i))
=0.
\]
Since $\psi$ is an immersion on each $G$-orbit, $(d\psi)_p$ is injective on
$\operatorname{Im}(d\sigma_p)_e$.
Hence, 
$\sum_{i=1}^n t_iX_i(p)=0$.
By the pointwise linear independence of $X_1,\dots,X_n$ on $M$, it follows that
$t_1=\cdots=t_n=0$.
Therefore, $Y_1,\dots,Y_n$ are pointwise linearly independent on $N$. Consequently,
$\rho_{G,\mathfrak{s}}(M,\sigma)\leq \rho_{G,\mathfrak{s}}(N,\tau)$.
Thus, $\rho_{G,\mathfrak{s}}$ is well-defined as a functor.
\end{proof}

Let
$\eta : GL(N,\mathbb{C}) \to GL(2N,\mathbb{R})$, 
$A+iB \mapsto
\begin{pmatrix}
A & -B\\
B & A
\end{pmatrix}$,
and let $\varphi : G \to GL(N,\mathbb{C})$ be a Lie group homomorphism.
Since $(\mathbb C^N\backslash\{0\},\varphi)$ and
$(\mathbb R^{2N}\backslash\{0\},\eta\circ\varphi)$ are isomorphic as
$G$-manifolds, Proposition~\ref{p34} implies that
\[
\rho_{G,\mathfrak{s}}(\mathbb{C}^N\backslash\{0\},\varphi)
=
\rho_{G,\mathfrak{s}}(\mathbb{R}^{2N}\backslash\{0\},\eta\circ\varphi).
\]

\subsection{Geometric interpretation of $\rho^{(2)}(\mathfrak{g}, \iota)$}

In this subsection, we discuss the relationship between $\rho_{G,\mathfrak{s}}(M,\sigma)$ and $\rho^{(2)}(\mathfrak{g},\iota)$ defined by Kannaka--Tojo.

Let $\mathfrak g$ be a real Lie algebra, let
$\mathfrak{s}\subset\mathfrak g$ be a vector subspace, and let
$\iota:\mathfrak g\to\mathfrak{gl}(N,\mathbb D)$ $(\mathbb D=\mathbb R\ \text{or}\ \mathbb C)$
be a representation of $\mathfrak g$.
\begin{Def}
    We define two integers $\rho_\mathfrak{s}^{(i)}(\mathfrak{g}, \iota )$ $(i = 1, 2)$ as the largest $n \in \mathbb{N}$ for which there exists an $\mathbb{R}$-linear map $f:\mathbb{R}^n\to \mathfrak{s}$ such that
    \begin{description}
\item[$\bullet$]$\rho_\mathfrak{s}^{(1)}(\mathfrak{g}, \iota)$ : $\iota(f(v))^2 = \| v \|^2I_N$ for any $v \in \mathbb{R}^n$,
\item[$\bullet$]$\rho_\mathfrak{s}^{(2)}(\mathfrak{g}, \iota)$ : $\iota(f(v))$ is invertible for any non-zero $v \in \mathbb{R}^n$.
 \end{description}
 Here, $\| v \|$ denotes the standard norm of $v \in \mathbb{R}^n$. In particular, suppose that $\iota:\mathfrak g\to\mathfrak{gl}(N,\mathbb C)$ is faithful and that $\iota(\mathfrak g)$ is self-adjoint. If $\mathfrak{g}=\mathfrak{k}+\mathfrak{p}$ is the associated Cartan decomposition and $\mathfrak{s}=\mathfrak{p}$, then
$$
\rho^{(i)}_{\mathfrak{s}}(\mathfrak{g},\iota)=\rho^{(i)}(\mathfrak{g},\iota).
$$
\end{Def}

\begin{Ex}
Let $N\geq2$, let $\mathfrak g=\mathfrak{sl}(N,\mathbb R)$, and let
$\iota:\mathfrak g\to\mathfrak{gl}(N,\mathbb R)$ be the natural
representation. Set $\mathfrak s=\mathfrak g$. We claim that
\[
\rho^{(2)}_{\mathfrak s}(\mathfrak g,\iota)=\rho(N).
\]
\begin{proof}
By Adams--Lax--Phillips
\cite{AdamsLaxPhillips65,AdamsLaxPhillips66}, the classical upper
bound for nonsingular real matrix spaces gives
\[
\rho^{(2)}_{\mathfrak s}(\mathfrak g,\iota)\leq\rho(N).
\]
By Au-Yeung \cite{AuYik71}, there exist
$B_1,\dots,B_{\rho(N)-1}\in\mathfrak o(N)$ such that
\[
\sum_{i=1}^{\rho(N)-1}t_iB_i
\]
is invertible for every
$(t_1,\dots,t_{\rho(N)-1})\in
\mathbb R^{\rho(N)-1}\backslash\{0\}$. Put
\[
C\coloneqq\operatorname{diag}(1,\dots,1,-(N-1)),\qquad
A_1\coloneqq C,\qquad
A_i\coloneqq B_{i-1}C\quad(2\leq i\leq\rho(N)).
\]
We have $\operatorname{tr}(A_1)=0$. Moreover, since $C^\top=C$ and
$B_{i-1}^\top=-B_{i-1}$,
\[
\operatorname{tr}(A_i)
=\operatorname{tr}((B_{i-1}C)^\top)
=-\operatorname{tr}(CB_{i-1})
=-\operatorname{tr}(B_{i-1}C)=0
\]
for $2\leq i\leq\rho(N)$. Hence $A_i\in\mathfrak{sl}(N,\mathbb R)$
for every $i$. For $(t_1,\dots,t_{\rho(N)})\neq0$, we have
\[
\sum_{i=1}^{\rho(N)}t_iA_i
=\left(t_1I_N+\sum_{i=2}^{\rho(N)}t_iB_{i-1}\right)C.
\]
If $t_1=0$, the first factor is invertible by the choice of the $B_i$.
If $t_1\neq0$, put
$S\coloneqq\sum_{i=2}^{\rho(N)}t_iB_{i-1}\in\mathfrak o(N)$.
Let $x\in\mathbb R^N$ and suppose that $(t_1I_N+S)x=0$. Then
\[
0=\langle Sx,x\rangle=-t_1\lVert x\rVert^2,
\]
and hence $x=0$. Thus $t_1I_N+S$ is invertible. Since $C$ is also
invertible, $\sum_{i=1}^{\rho(N)}t_iA_i$ is invertible for every
$(t_1,\dots,t_{\rho(N)})\in\mathbb R^{\rho(N)}\backslash\{0\}$.
Therefore,
\[
\rho^{(2)}_{\mathfrak s}(\mathfrak g,\iota)\geq\rho(N).
\]
\end{proof}

For comparison, consider the standard Cartan decomposition
\[
\mathfrak g=\mathfrak{so}(N)+\mathfrak p,
\qquad
\mathfrak p=\{A\in\mathfrak{sl}(N,\mathbb R)\mid A^\top=A\}.
\]
Since $\mathfrak s=\mathfrak g\neq\mathfrak p$, this is a non-trivial
example of the generalized invariant with
$\mathfrak s\neq\mathfrak p$. The computation of Kannaka--Tojo gives
\[
\rho_{\mathfrak p}^{(2)}
(\mathfrak{sl}(N,\mathbb R),\iota)
=
\begin{cases}
\rho(N/2)+1 & \text{if $N$ is even},\\
1 & \text{if $N$ is odd}.
\end{cases}
\]
Thus, the values of $\rho_{\mathfrak s}^{(2)}(\mathfrak g,\iota)$ for
$\mathfrak s=\mathfrak g$ and $\mathfrak s=\mathfrak p$ need not
coincide.
\end{Ex}

Let $\eta:GL(N, \mathbb{C}) \to GL(2N, \mathbb{R}), A + iB \mapsto \begin{pmatrix}
   A & -B \\
   B & A
\end{pmatrix}$.
\begin{Lem}\label{lem:realification}
For a representation
$\iota:\mathfrak g\to\mathfrak{gl}(N,\mathbb C)$ and $i=1,2$, we have
$$
\rho^{(i)}_\mathfrak{s}
\bigl(\mathfrak{g},(d\eta)_{I_N}\circ\iota\bigr)
=\rho^{(i)}_\mathfrak{s}(\mathfrak{g},\iota).
$$
\end{Lem}
\begin{proof}
Put $R\coloneqq(d\eta)_{I_N}$. Since $R$ is an injective unital
$\mathbb R$-algebra homomorphism, for every
$Z\in\mathfrak{gl}(N,\mathbb C)$ and $c\in\mathbb R$,
\[
Z^2=cI_N
\quad\Longleftrightarrow\quad
R(Z)^2=cI_{2N}.
\]
Moreover, $Z$ is invertible if and only if $R(Z)$ is invertible.
Hence the same
$\mathbb R$-linear maps $f:\mathbb R^n\to\mathfrak s$ satisfy the
defining conditions on both sides.
\end{proof}

We now return to real representations. Let $G$ be a Lie group, let
$\mathfrak{s}\subset\mathfrak g$ be a vector subspace, let
$\varphi:G\to GL(N,\mathbb R)$ be a Lie group homomorphism, and put
$\iota\coloneqq(d\varphi)_e$. Also, for each
$x\in\mathbb R^N\backslash\{0\}$, put
$\varphi_x:G\to\mathbb R^N\backslash\{0\}$,
$g\mapsto\varphi(g)x$.

As one of the main results, we obtain the following theorem for real
representations. It follows from Lemmas~\ref{l310} and~\ref{l311}
below.
\begin{Thm}\label{t36}
The following holds:
    $$
    \rho_{G, \mathfrak{s}}(\mathbb{R}^N\backslash\{0\}, \varphi) = \rho^{(2)}_\mathfrak{s}(\mathfrak{g}, \iota).
    $$
\end{Thm}

\begin{Cor}\label{c37}
Let $\mathfrak g$ be a real reductive Lie algebra, let
$\iota:\mathfrak g\to\mathfrak{gl}(N,\mathbb C)$
be a faithful representation, and assume that $\iota(\mathfrak g)$ is
self-adjoint. Let
$\mathfrak g=\mathfrak k+\mathfrak p$
be the associated Cartan decomposition. Take a simply connected Lie
group $G$ and a Lie group homomorphism
$\varphi:G\to GL(N,\mathbb C)$
such that
$\iota=(d\varphi)_e$. Then
$$
\rho_{G,\mathfrak p}(\mathbb{C}^{N}\backslash\{0\}, \varphi) 
=\rho_{G,\mathfrak p}(\mathbb{R}^{2N}\backslash\{0\}, \eta\circ \varphi) 
= \rho^{(2)}(\mathfrak{g}, \iota).
$$
\end{Cor}
\begin{proof}
The first equality follows from the $G$-equivariant diffeomorphism
$F:\mathbb C^N\backslash\{0\}\to\mathbb R^{2N}\backslash\{0\}$ above
and Proposition~\ref{p34}. Applying Theorem~\ref{t36} to the real
representation $\eta\circ\varphi$ gives
\[
\rho_{G,\mathfrak p}(\mathbb R^{2N}\backslash\{0\},\eta\circ\varphi)
=\rho^{(2)}_{\mathfrak p}
\bigl(\mathfrak g,(d(\eta\circ\varphi))_e\bigr).
\]
By Lemma~\ref{lem:realification}, the right-hand side is
$\rho^{(2)}_{\mathfrak p}(\mathfrak g,\iota)$, which equals
$\rho^{(2)}(\mathfrak g,\iota)$ by Definition~3.5.
\end{proof}

\begin{Cor}
Under the assumptions of Corollary~\ref{c37}, let $n\in\mathbb N$. Then the following are equivalent.
\begin{description}
\item[(1)] $n \leq \rho^{(2)}(\mathfrak{g}, \iota)$,
\item[(2)] There exist $n$ pointwise linearly independent
$(G,\mathfrak p)$-fundamental vector fields on
$\mathbb{C}^{N}\backslash\{0\}$.
\end{description}
\end{Cor}
\begin{proof}
This follows immediately from Corollary~\ref{c37} and the definition
of $\rho_{G,\mathfrak p}$.
\end{proof}

Therefore, $\rho^{(2)}(\mathfrak g,\iota)$ is interpreted in terms of the pointwise linear independence of fundamental vector fields.

\begin{Lem}\label{l39}
    Fix $x\in \mathbb{R}^N\backslash\{0\}$. Then, for any $A\in T_eG$,
    $$
    (d\varphi_x)_e(A) = (d\varphi)_e(A)x.
    $$
\end{Lem}
\begin{proof}
    Fix $x\in \mathbb{R}^N\backslash\{0\}$, and define
    $
    \psi_x:GL(N, \mathbb{R}) \to \mathbb{R}^{N}\backslash\{0\}, B \mapsto Bx.
    $
    Since $\varphi_x = \psi_x\circ \varphi$, we have
    \begin{equation*}
\xymatrix{
T_eG\ar[r]^-{(d\varphi)_e}\ar[rd]_-{(d\varphi_x)_e}\ar@{}@<2.0ex>[rd]|{\circlearrowright}&T_{I_N}(GL(N, \mathbb{R}))\ar[d]^-{(d\psi_x)_{I_N}}\\
&T_x(\mathbb{R}^N\backslash\{0\})
}
\end{equation*}
Hence, for any $A\in T_eG$, we have
\begin{eqnarray*}
     (d\varphi_x)_e(A)
     &=& ((d\psi_x)_{I_N}\circ(d\varphi)_e)(A)
     = (d\psi_x)_{I_N}((d\varphi)_e(A)) .
\end{eqnarray*}
Here, we identify $GL(N,\mathbb{R})$ with an open subset of $\mathbb{R}^{N^2}$ via
\[
(a_{ij})_{1\le i,j\le N}
\longmapsto
(a_{11},a_{12},\dots,a_{1N},a_{21},\dots,a_{NN})^{\top}.
\]
With respect to this coordinate system, the matrix of $(d\psi_x)_{I_N}$ is the following $N\times N^2$ matrix:
$$
\begin{pmatrix}
   x^{\top} & 0 & 0&0\\
   0& x^{\top}&0&0\\
    0&0 & \ddots&\vdots\\
   0&0&\cdots&x^{\top}
\end{pmatrix}.
$$
Therefore,
$$
    (d\varphi_x)_e(A) = (d\varphi)_e(A)x.
    $$
\end{proof}

\begin{Lem}\label{l310}
   For $A_1, A_2, ..., A_n \in \mathfrak{s}$, the following are equivalent.
\begin{description}
\item[(1)]The fundamental vector fields $X_1, X_2, ..., X_n$ corresponding to $A_1, A_2, ..., A_n$ are pointwise linearly independent on $\mathbb{R}^N\backslash\{0\}$.
\item[(2)]For each $(t_1, t_2, ..., t_n) \in \mathbb{R}^n\backslash\{0\}$, $(d\varphi)_e(\sum_{i=1}^{n}t_iA_i)$ is invertible.
 \end{description}
\end{Lem}
\begin{proof}
    First, we show (1)$\Rightarrow$(2). 
    By the pointwise linear independence of the fundamental vector fields corresponding to $A_1, A_2, ..., A_n$, for each $x \in \mathbb{R}^{N}\backslash\{0\}$, $(d\varphi_x)_e(A_1), (d\varphi_x)_e(A_2), ..., (d\varphi_x)_e(A_n)$ are linearly independent. Hence for each $t=(t_1, t_2, ..., t_n) \in \mathbb{R}^{n}\backslash\{0\}$, we have $\sum_{i=1}^nt_i(d\varphi_x)_e(A_i) \neq 0$. Put $\tilde{A}_t \coloneqq \sum_{i=1}^nt_i(d\varphi_x)_e(A_i)$. Then by Lemma \ref{l39},
\begin{eqnarray*}
     \sum_{i=1}^nt_i(d\varphi_x)_e(A_i)
     &=& (d\varphi_x)_e(\sum_{i=1}^nt_iA_i) = \tilde{A}_tx \neq 0.
\end{eqnarray*}
Thus, considering
$
\tilde{A}_t:\mathbb{R}^N \to \mathbb{R}^N, x \mapsto \tilde{A}_tx,
$
we obtain $\textrm{Ker}\tilde{A}_t = \{0\}$, so $\tilde{A}_t$ is injective. Therefore, $\tilde{A}_t$ is invertible.

    Next, we show (2)$\Rightarrow$(1). For each $i \in \{ 1, 2, ..., n\}$, let $X_i$ be the fundamental vector field corresponding to $A_i$. Fix $x\in \mathbb{R}^N\backslash\{0\}$. Suppose that, for $t = (t_1, t_2, ..., t_n) \in \mathbb{R}^n$,
    $$
    \sum_{i=1}^nt_iX_i(x) = \sum_{i=1}^{n}t_i(d\varphi_x)_e(A_i) = 0.
    $$
    If $t\neq 0$, then
$(d\varphi)_e(\sum_{i=1}^n t_iA_i)$
is invertible. On the other hand,
\begin{align*}
(d\varphi)_e(\sum_{i=1}^n t_iA_i)x
&= (d\varphi_x)_e(\sum_{i=1}^n t_iA_i) \\
&= \sum_{i=1}^n t_i(d\varphi_x)_e(A_i)
=0,
\end{align*}
which is a contradiction. Hence, $t=0$.
\end{proof}

\begin{Lem}\label{l311}
    Let $n\in\mathbb N$. The following are equivalent.
    \begin{description}
\item[(1)]$n \leq \rho^{(2)}_\mathfrak{s}(\mathfrak{g}, \iota)$.
\item[(2)]There exist $A_1, ..., A_n \in \mathfrak{s}$ such that for each $(t_1, ... ,t_n) \in \mathbb{R}^n\backslash\{0\}$, $(d\varphi)_e(\sum_{i=1}^nt_iA_i)$ is invertible.
 \end{description}
\end{Lem}
\begin{proof}
First, we show \((1)\Rightarrow(2)\).
Let
$f:\mathbb{R}^n\to \mathfrak{s}$
be an $\mathbb{R}$-linear map such that, for every
$v\in \mathbb{R}^n\backslash\{0\}$,
the matrix
$((d\varphi)_e\circ f)(v)$
is invertible. Let \(e_1,\dots,e_n\) be the standard basis of \(\mathbb{R}^n\), and put
$A_i:=f(e_i)\in \mathfrak{s}$ for each $i \in \{1, ..., n\}$.
Then, for each
$(t_1,\dots,t_n)\in \mathbb{R}^n\backslash\{0\}$,
we have
\[
(d\varphi)_e(\sum_{i=1}^n t_iA_i)
=
(d\varphi)_e(\sum_{i=1}^n t_if(e_i))
=
(d\varphi)_e(f(\sum_{i=1}^n t_ie_i)).
\]
Since
$\sum_{i=1}^n t_ie_i\neq 0$,
it follows that
$(d\varphi)_e(\sum_{i=1}^n t_iA_i)$
is invertible.

Next, we show \((2)\Rightarrow(1)\).
Assume that there exist \(A_1,\dots,A_n\in \mathfrak{s}\) such that, for each
$(t_1,\dots,t_n)\in \mathbb{R}^n\backslash\{0\}$,
the matrix
$(d\varphi)_e(\sum_{i=1}^n t_iA_i)$
is invertible.
Define an $\mathbb{R}$-linear map
$f:\mathbb{R}^n\to \mathfrak{s}$
by
$f(e_i)=A_i
$ for each $i \in \{1, ..., n\}$.
Then condition (1) holds.
\end{proof}

\begin{proof}[Proof of the first assertion of Theorem~\ref{t13}]
Let
\[
F:\mathbb C^N\backslash\{0\}\longrightarrow
\mathbb R^{2N}\backslash\{0\},\qquad x+iy\longmapsto(x,y).
\]
Then $F$ is a $G$-equivariant diffeomorphism from the action defined by
$\varphi$ to the action defined by $\eta\circ\varphi$. Therefore,
Proposition~\ref{p34} gives
\[
\rho_{G,\mathfrak p}(\mathbb C^N\backslash\{0\},\varphi)
=
\rho_{G,\mathfrak p}(\mathbb R^{2N}\backslash\{0\},\eta\circ\varphi).
\]
Applying Theorem~\ref{t36} to $\eta\circ\varphi$ and then applying
Lemma~\ref{lem:realification}, we obtain
\begin{align*}
\rho_{G,\mathfrak p}(\mathbb R^{2N}\backslash\{0\},\eta\circ\varphi)
&=\rho^{(2)}_{\mathfrak p}
\bigl(\mathfrak g,(d(\eta\circ\varphi))_e\bigr)\\
&=\rho^{(2)}_{\mathfrak p}(\mathfrak g,\iota)
=\rho^{(2)}(\mathfrak g,\iota).
\end{align*}
This proves the first assertion.
\end{proof}

\begin{proof}[Proof of Theorem~\ref{t14}]
Let $\mathfrak{so}(N,N)=\mathfrak{k}+\mathfrak{p}$ be a Cartan decomposition, and let
$\varphi:SO(N,N)\to GL(2N,\mathbb{R})$
be the standard representation. Set $\iota \coloneqq (d\varphi)_e$. Then, 
\begin{align*}
\rho(N)
&\overset{\text{\cite{KannakaTojo26}}}{=}
\rho^{(2)}(\mathfrak{so}(N, N), \iota)
\overset{\text{Thm.~\ref{t36}}}{=}
\rho_{SO(N, N), \mathfrak{p}}(\mathbb{R}^{2N}\backslash\{0\}, \varphi)\\
&\coloneqq \max\left\{
n\in\mathbb N \,\middle|\,
\begin{array}{l}
\text{there exist } n \textrm{ }(SO(N, N), \mathfrak{p}) \text{-fundamental vector fields on }\mathbb{R}^{2N}\backslash\{0\}\\
\text{which are linearly independent at every point}
\end{array}
\right\}.
\end{align*}

Let
$\varphi':O(N)\to GL(N,\mathbb{R})$
be the standard representation, and set $\iota' \coloneqq (d\varphi')_e$. Then, 
\begin{align*}
\rho(N)-1
&\overset{\text{\cite{AndreiUwe11}}}{=}
\max\left\{
n\in\mathbb N \,\middle|\,
\begin{array}{l}
\text{there exists an $\mathbb{R}$-linear map } f:\mathbb{R}^n \to \mathfrak{o}(N) \text{ such that}\\
\text{for every } v \in \mathbb{R}^n \backslash \{0\},\ \iota'(f(v)) \text{ is invertible}
\end{array}
\right\}\\
&\eqqcolon
\rho^{(2)}_{\mathfrak{o}(N)}(\mathfrak{o}(N), \iota')
\overset{\text{Thm.~\ref{t36}}}{=}
\rho_{O(N), \mathfrak{o}(N)}(\mathbb{R}^{N}\backslash\{0\}, \varphi')\\
&\coloneqq \max\left\{
n\in\mathbb N \,\middle|\,
\begin{array}{l}
\text{there exist } n \textrm{ }(O(N), \mathfrak{o}(N)) \text{-fundamental vector fields on }\mathbb{R}^{N}\backslash\{0\}\\
\text{which are linearly independent at every point}
\end{array}
\right\}.
\end{align*}
Therefore, the assertion of the theorem follows.
\end{proof}

\section{Hurwitz--Radon numbers for $G$-manifolds and their affine connections}
In this section, we define $\rho^{\pm}_{G,\mathfrak{s}}(M,\sigma,\nabla)$ for G-manifolds and their affine connections, and show that $\rho^{-}_{G,\mathfrak{s}}(M,\sigma,\nabla)$ is an extension of $\rho^{(1)}(\mathfrak g,\iota)$ defined by Kannaka--Tojo. This provides a geometric interpretation of $\rho^{(1)}(\mathfrak g,\iota)$. Furthermore, we show a relationship between $\rho^{+}_{G,\mathfrak{s}}(M,\sigma,\nabla)$ and Clifford structures.

\subsection{Definition and properties of $\rho_{G, \mathfrak{s}}(M, \sigma, \nabla)$}
Let $G$ be a Lie group, let $\mathfrak{s}\subset\mathfrak g$ be a vector subspace, let $M$ be a $G$-manifold with group action $\sigma$, and let
$\nabla:\mathfrak X(M)\times\mathfrak X(M)\to\mathfrak X(M)$
be an affine connection on $M$.
We define the $\mathbb R$-subspace
\[
\nabla^{M}\mathfrak{s}
\coloneqq
\{\nabla X_A \mid A\in \mathfrak{s}\}
\subset \operatorname{End}_{C^\infty(M)}(\mathfrak X(M)),
\]
where $X_A$ denotes the fundamental vector field on $M$ corresponding to $A\in\mathfrak g$.

Motivated by the proof of Fact~\ref{f23} due to Kannaka--Tojo, we introduce the following integer for a $G$-manifold equipped with an affine connection.
\begin{Def}
    For each $\epsilon\in\{\pm1\}$, we define integer
$\rho^\epsilon_{G,\mathfrak{s}}(M,\sigma,\nabla)$
as the largest $n\in\mathbb N$ for which there exists a unital $\mathbb R$-algebra homomorphism
$f:Cl_n^\epsilon\to \operatorname{End}_{C^\infty(M)}(\mathfrak X(M))$ 
such that
\begin{itemize}
\item $f(e_i)\in\nabla^M \mathfrak{s}$ for any $i \in \{1, ..., n\}$.
\end{itemize}
Here, the notation $+$ (resp. $-$) corresponds to $\epsilon=1$ (resp. $\epsilon=-1$), and
$Cl_n^+=Cl_n$, $Cl_n^-=Cl_{0,n}$.
\end{Def}

First, we discuss the functoriality of $\rho_{G,\mathfrak{s}}^{\pm}$.
\begin{Def}
We define a category $G\textrm{-}\mathbf{MfdConn}$ as follows.

\begin{itemize}
\item An object of $G\textrm{-}\mathbf{MfdConn}$ is a triple $(M,\sigma,\nabla)$, where $M$ is a smooth manifold,
$\sigma:G\times M\to M$
is a smooth $G$-action, and $\nabla$ is an affine connection on $M$.

\item A morphism
$\psi:(M,\sigma,\nabla)\to (N,\tau,\nabla')$
is a smooth map $\psi:M\to N$ satisfying the following conditions:
\begin{enumerate}
\item $\psi$ is an immersion;
\item $\psi$ is $G$-equivariant;
\item for each $p\in M$, $v\in T_pM$, and $A\in \mathfrak{s}$,
\[
(d\psi)_p(\nabla_vX_A)
=
\nabla'_{(d\psi)_p(v)}Y_A.
\]
Here, $X_A$ and $Y_A$ denote the fundamental vector fields corresponding to $A$ with respect to $\sigma$ and $\tau$, respectively.
\end{enumerate}
\end{itemize}
\end{Def}

\begin{Prop}\label{p43}
$\rho_{G,\mathfrak{s}}^{\pm}$ are contravariant functors from $G\textrm{-}\mathbf{MfdConn}$ to the partially ordered set  $(\mathbb{Z}_{\geq 0},\leq)$.
\end{Prop}
\begin{proof}
Fix $\epsilon\in\{\pm1\}$. First, we show that $\rho_{G,\mathfrak{s}}^{\epsilon}$ is well-defined as an assignment on morphisms. Let
$(M,\sigma,\nabla),\ (N,\tau,\nabla')\in \mathrm{Ob}(G\textrm{-}\mathbf{MfdConn})$,
and let
$\psi:(M,\sigma,\nabla)\to (N,\tau,\nabla')$
be a morphism. It suffices to show that
$\rho^{\epsilon}_{G,\mathfrak{s}}(N,\tau,\nabla')\leq \rho^{\epsilon}_{G,\mathfrak{s}}(M,\sigma,\nabla)$.
Put
$n\coloneqq \rho^{\epsilon}_{G,\mathfrak{s}}(N,\tau,\nabla')$.
Then there exists a unital $\mathbb{R}$-algebra homomorphism
$f':Cl_n^{\epsilon}\to \mathrm{End}_{C^\infty(N)}(\mathfrak{X}(N))$
such that, for each $i\in\{1,\dots,n\}$,
$f'(e_i)\in \nabla'^N \mathfrak{s}$.
Hence, for each $i\in\{1,\dots,n\}$, there exists $A_i\in \mathfrak{s}$ such that
$f'(e_i)=\nabla'Y_{A_i}$,
where $Y_{A_i}$ is the fundamental vector field corresponding to $A_i$ with respect to $\tau$.
Define an $\mathbb{R}$-linear map
$f:Cl_n^{\epsilon}\to \mathrm{End}_{C^\infty(M)}(\mathfrak{X}(M))$
by
$f(e_i)\coloneqq \nabla X_{A_i}$ for each $i \in \{1, ..., n\}$,
where $X_{A_i}$ is the fundamental vector field corresponding to $A_i$ with respect to $\sigma$.
Then, for all $i,j\in\{1,\dots,n\}$, $p\in M$, and $v\in T_pM$,
\begin{align*}
(d\psi)_p\bigl((f(e_i)f(e_j)+f(e_j)f(e_i))(v)\bigr)
&=(d\psi)_p\bigl((\nabla X_{A_i}\circ\nabla X_{A_j}
+\nabla X_{A_j}\circ\nabla X_{A_i})(v)\bigr)\\
&=(d\psi)_p\bigl(\nabla_{\nabla_vX_{A_j}}X_{A_i}
+\nabla_{\nabla_vX_{A_i}}X_{A_j}\bigr)\\
&=(d\psi)_p(\nabla_{\nabla_vX_{A_j}}X_{A_i})
+(d\psi)_p(\nabla_{\nabla_vX_{A_i}}X_{A_j})\\
&=\nabla'_{(d\psi)_p(\nabla_vX_{A_j})}Y_{A_i}
+\nabla'_{(d\psi)_p(\nabla_vX_{A_i})}Y_{A_j}\\
&=(\nabla'Y_{A_i}\circ\nabla'Y_{A_j}
+\nabla'Y_{A_j}\circ\nabla'Y_{A_i})((d\psi)_p(v))\\
&=\epsilon\,2\delta_{ij}(d\psi)_p(v).
\end{align*}
Therefore,
\[
f(e_i)f(e_j)+f(e_j)f(e_i)=\epsilon\,2\delta_{ij}\operatorname{id}_{TM}.
\]
Hence, $f$ is a unital $\mathbb{R}$-algebra homomorphism. Thus,
$\rho^{\epsilon}_{G,\mathfrak{s}}(N,\tau,\nabla')\leq \rho^{\epsilon}_{G,\mathfrak{s}}(M,\sigma,\nabla)$.
It is clear that $\rho_{G,\mathfrak{s}}^{\epsilon}$ satisfies compatibility with identities and compositions, and hence is a contravariant functor from $G\textrm{-}\mathbf{MfdConn}$ to $(\mathbb{Z}_{\geq 0} ,\leq)$.
\end{proof}

Let
$F: \mathbb{C}^{N}\backslash\{0\} \to \mathbb{R}^{2N}\backslash\{0\}, 
x + iy \mapsto
\begin{pmatrix}
x\\
y
\end{pmatrix}$, 
let $\eta : GL(N, \mathbb{C}) \to GL(2N, \mathbb{R}),
A + iB \mapsto
\begin{pmatrix}
A & -B \\
B & A
\end{pmatrix}$, 
and let $\varphi:G\to GL(N,\mathbb{C})$ be a Lie group homomorphism.
Let $\nabla$ be the standard flat affine connection on
$\mathbb{C}^{N}\backslash\{0\}$, and let $F_*\nabla$ be the standard
flat affine connection on $\mathbb{R}^{2N}\backslash\{0\}$ obtained by
pushforward under the isomorphism $F$.
Since $F$ is an isomorphism in $G\textrm{-}\mathbf{MfdConn}$,
Proposition~\ref{p43} implies that
\[
\rho^{\pm}_{G,\mathfrak{s}}(\mathbb{C}^N\backslash\{0\},\varphi,\nabla)
=
\rho^{\pm}_{G,\mathfrak{s}}(\mathbb{R}^{2N}\backslash\{0\},\eta\circ\varphi,F_*\nabla).
\]

\subsection{Geometric interpretation of $\rho^{(1)}(\mathfrak{g}, \iota)$}
Let $G$ be a Lie group, let $\mathfrak{s} \subset \mathfrak{g}$ be a vector subspace, and let $\varphi:G \to GL(N, \mathbb{R})$ be a Lie group homomorphism. Put $\iota \coloneqq (d\varphi)_e$. Below, consider the standard flat affine connection $\nabla$ on $\mathbb{R}^N\backslash\{0\}$ defined by
$$
\nabla:\mathfrak{X}(\mathbb{R}^N\backslash\{0\}) \times \mathfrak{X}(\mathbb{R}^N\backslash\{0\}) \to \mathfrak{X}(\mathbb{R}^N\backslash\{0\}),
(X,\sum_{i=1}^{N}f_i\frac{\partial}{\partial x_i})\mapsto \sum_{k=1}^{N}Xf_k\frac{\partial}{\partial x_k}.
$$

As one of the main results, we obtain the following theorem for real
representations. It follows from Lemmas~\ref{l48}, \ref{l411},
and~\ref{l412} below.
\begin{Thm}\label{t44}
The following holds:
    $$
    \rho^{-}_{G, \mathfrak{s}}(\mathbb{R}^N\backslash\{0\}, \varphi, \nabla) = \rho^{(1)}_\mathfrak{s}(\mathfrak{g}, \iota).
    $$
\end{Thm}

By Theorem~\ref{t44}, we obtain the following.

\begin{Cor}\label{c45}
Let $\mathfrak g$ be a real reductive Lie algebra, let
$\iota:\mathfrak g\to\mathfrak{gl}(N,\mathbb C)$
be a faithful representation, and assume that $\iota(\mathfrak g)$ is
self-adjoint. Let
$\mathfrak g=\mathfrak k+\mathfrak p$
be the associated Cartan decomposition. Take a simply connected Lie group $G$ and a Lie group homomorphism
$\varphi:G\to GL(N,\mathbb C)$
such that
$\iota=(d\varphi)_e$. Let $\nabla$ be the standard flat affine connection on $\mathbb{C}^{N}\backslash\{0\}$. Then the following holds:
$$
\rho^{-}_{G,\mathfrak p}(\mathbb{C}^{N}\backslash\{0\},\varphi,\nabla)
=
\rho^{-}_{G,\mathfrak p}(\mathbb{R}^{2N}\backslash\{0\},\eta\circ\varphi,F_*\nabla)
=
\rho^{(1)}(\mathfrak g,\iota).
$$
\end{Cor}
\begin{proof}
The first equality is the realification equality established above for
$G$-manifolds with affine connections. Applying Theorem~\ref{t44} to
the real representation $\eta\circ\varphi$ gives
\[
\rho^{-}_{G,\mathfrak p}
(\mathbb R^{2N}\backslash\{0\},\eta\circ\varphi,F_*\nabla)
=\rho^{(1)}_{\mathfrak p}
\bigl(\mathfrak g,(d(\eta\circ\varphi))_e\bigr).
\]
By Lemma~\ref{lem:realification} and Definition~3.5, the right-hand
side is $\rho^{(1)}(\mathfrak g,\iota)$.
\end{proof}

\begin{Cor}\label{c46}
Under the assumptions of Corollary~\ref{c45}, let $n\in\mathbb N$. Then the following are equivalent.
\begin{description}
\item[(1)] $n\leq \rho^{(1)}(\mathfrak g,\iota)$.
\item[(2)] There exists a unital $\mathbb R$-algebra homomorphism
$
f:Cl_{0,n}\to
\operatorname{End}_{C^\infty(\mathbb{C}^{N}\backslash\{0\})}
(\mathfrak X(\mathbb{C}^{N}\backslash\{0\}))
$
such that $f(e_i)\in\nabla\mathfrak p$ for each $i\in\{1,\dots,n\}$.
\end{description}
\end{Cor}
\begin{proof}
This follows immediately from Corollary~\ref{c45} and the definition
of $\rho^{-}_{G,\mathfrak p}$.
\end{proof}

Therefore, $\rho^{(1)}(\mathfrak g,\iota)$ can be interpreted in the framework of $G$-manifolds and their affine connections.

Moreover, by Theorem~\ref{t36}, we obtain the following.
\begin{Cor}
    Let $\varphi:G \to GL(N, \mathbb{R})$ be a Lie group homomorphism, and let $\nabla$ be the standard flat affine connection on $\mathbb{R}^N\backslash\{0\}$. Then
    $$
    \rho^{\pm}_{G, \mathfrak{s}}(\mathbb{R}^N\backslash\{0\}, \varphi, \nabla) \leq \rho_{G, \mathfrak{s}}(\mathbb{R}^N\backslash\{0\}, \varphi).
    $$
\end{Cor}

\begin{Lem}\label{l48}
The following are equivalent.
\begin{description}
\item[(1)] $n \leq \rho^{(1)}_\mathfrak{s}(\mathfrak{g}, \iota)$.
\item[(2)] There exist $A_1,\dots,A_n\in \mathfrak{s}$ such that, for each $i,j\in\{1,\dots,n\}$,
$$
(d\varphi)_e(A_i)(d\varphi)_e(A_j)+(d\varphi)_e(A_j)(d\varphi)_e(A_i)=2\delta_{ij}I_N.
$$
\end{description}
\end{Lem}
\begin{proof}
First, we show $(1)\Rightarrow(2)$. Let
$f:\mathbb{R}^n\to \mathfrak{s}$
be an $\mathbb{R}$-linear map such that, for each $v\in\mathbb{R}^n\backslash\{0\}$,
$
(d\varphi)_e(f(v))^2=\|v\|^2I_N.
$
Let $e_1,\dots,e_n$ be the standard basis of $\mathbb{R}^n$, and put
$
A_i\coloneqq f(e_i)\in \mathfrak{s}$, $\tilde{A}_i\coloneqq (d\varphi)_e(A_i)
$
for each $i\in\{1,\dots,n\}$. Then, for each $v=\sum_{i=1}^n v_ie_i\in\mathbb{R}^n$,
$$
(d\varphi)_e(f(v))^2
=
\left(\sum_{i=1}^n v_i\tilde{A}_i\right)^2
=
\sum_{i=1}^n v_i^2\tilde{A}_i^2
+
\sum_{i<j} v_iv_j(\tilde{A}_i\tilde{A}_j+\tilde{A}_j\tilde{A}_i)
=
\|v\|^2I_N.
$$
Comparing coefficients, we obtain
$
\tilde{A}_i\tilde{A}_j+\tilde{A}_j\tilde{A}_i=2\delta_{ij}I_N.
$

Next, we show $(2)\Rightarrow(1)$. Suppose that $A_1,\dots,A_n\in \mathfrak{s}$ satisfy
$$
(d\varphi)_e(A_i)(d\varphi)_e(A_j)+(d\varphi)_e(A_j)(d\varphi)_e(A_i)=2\delta_{ij}I_N
$$
for each $i,j\in\{1,\dots,n\}$. Put
$\tilde A_i\coloneqq(d\varphi)_e(A_i)$ for each
$i\in\{1,\dots,n\}$. We define an $\mathbb{R}$-linear map
$f:\mathbb{R}^n\to \mathfrak{s}$
by
$
f(e_i)\coloneqq A_i
$
for each $i\in\{1,\dots,n\}$. Then, for each $v=\sum_{i=1}^n v_ie_i\in\mathbb{R}^n$,
$$
(d\varphi)_e(f(v))^2
=
\sum_{i=1}^n v_i^2\tilde{A}_i^2
+
\sum_{i<j} v_iv_j(\tilde{A}_i\tilde{A}_j+\tilde{A}_j\tilde{A}_i)
=
\|v\|^2I_N.
$$
Hence,
$
n\leq \rho^{(1)}_\mathfrak{s}(\mathfrak{g},\iota).
$
\end{proof}

\begin{Lem}\label{l49}
Let $A\in\mathfrak{g}$, and put
$
\tilde{A}\coloneqq (d\varphi)_e(A)=
\begin{pmatrix}
\tilde{a}_{11}&\cdots&\tilde{a}_{1N}\\
\vdots&\ddots&\vdots\\
\tilde{a}_{N1}&\cdots&\tilde{a}_{NN}
\end{pmatrix}
\in\mathfrak{gl}(N,\mathbb{R}).
$
Then, for each $l\in\{1,\dots,N\}$,
$$
\nabla_{\frac{\partial}{\partial x_l}}X_A
=
\sum_{k=1}^N\tilde{a}_{kl}\frac{\partial}{\partial x_k}.
$$
\end{Lem}
\begin{proof}
Put
$
X_A=\sum_{k=1}^N f_k^A\frac{\partial}{\partial x_k}.
$
By Lemma~\ref{l39}, for each
$
p=
\begin{pmatrix}
p_1\\
\vdots\\
p_N
\end{pmatrix}
\in\mathbb{R}^N\backslash\{0\},
$
we have
\begin{eqnarray*}
X_A(p) = \tilde{A}p
= \begin{pmatrix}
\sum_{m=1}^N\tilde{a}_{1m}p_m\\
\vdots\\
\sum_{m=1}^N\tilde{a}_{Nm}p_m
\end{pmatrix}
= \sum_{k=1}^N f_k^A(p)(\frac{\partial}{\partial x_k})_p.
\end{eqnarray*}
Hence, for each $k\in\{1,\dots,N\}$,
$
f_k^A(p)=\sum_{m=1}^N\tilde{a}_{km}p_m.
$
Therefore, for each $l\in\{1,\dots,N\}$,
\begin{eqnarray*}
(\nabla_{\frac{\partial}{\partial x_l}}X_A)(p)
&=&
(
\sum_{k=1}^N
\frac{\partial f_k^A}{\partial x_l}
\frac{\partial}{\partial x_k}
)(p)
=
\sum_{k=1}^N
\frac{\partial(\sum_{m=1}^N\tilde{a}_{km}x_m)}{\partial x_l}(p)
(\frac{\partial}{\partial x_k})_p\\
&=&
\sum_{k=1}^N
\tilde{a}_{kl}
(\frac{\partial}{\partial x_k})_p.
\end{eqnarray*}
Thus, for each $l\in\{1,\dots,N\}$,
$$
\nabla_{\frac{\partial}{\partial x_l}}X_A
=
\sum_{k=1}^N\tilde{a}_{kl}\frac{\partial}{\partial x_k}.
$$
\end{proof}

\begin{Lem}\label{l410}
    Let $A_i$, $A_j \in \mathfrak{g}$, and put
$\tilde{A_i} = \begin{pmatrix}
   \tilde{a}_{11}^i & \cdots &\tilde{a}_{1N}^i\\
   \vdots & \ddots & \vdots\\
   \tilde{a}_{N1}^i & \cdots & \tilde{a}_{NN}^i
\end{pmatrix} \coloneqq (d\varphi)_e(A_i)$,
and
$ \tilde{A_j} = \begin{pmatrix}
   \tilde{a}_{11}^j & \cdots &\tilde{a}_{1N}^j\\
   \vdots & \ddots & \vdots\\
   \tilde{a}_{N1}^j & \cdots & \tilde{a}_{NN}^j
\end{pmatrix} \coloneqq (d\varphi)_e(A_j) \in \mathfrak{gl}(N, \mathbb{R})$.
Then, for each $l \in \{ 1, ..., N \}$,
$$
(\nabla X_{A_i}\circ \nabla X_{A_j} + \nabla X_{A_j}\circ \nabla X_{A_i})(\frac{\partial}{\partial x_l}) = \sum_{k=1}^N\sum_{m=1}^N(\tilde{a}_{mk}^i\tilde{a}_{kl}^j + \tilde{a}_{mk}^j\tilde{a}_{kl}^i)\frac{\partial}{\partial x_m}.
$$
\end{Lem}
\begin{proof}
For each $l \in \{ 1, ..., N \}$, by Lemma~\ref{l49},
    \begin{eqnarray*}
       (\nabla X_{A_i}\circ \nabla X_{A_j} + \nabla X_{A_j}\circ \nabla X_{A_i})(\frac{\partial}{\partial x_l})
       &=&(\nabla X_{A_i})(\sum_{k=1}^N\tilde{a}_{kl}^j\frac{\partial}{\partial x_k}) + (\nabla X_{A_j})(\sum_{k=1}^N\tilde{a}_{kl}^i\frac{\partial}{\partial x_k})\\
       &=&\sum_{k=1}^N\tilde{a}_{kl}^j\nabla_{\frac{\partial}{\partial x_k}} X_{A_i} + \sum_{k=1}^N\tilde{a}_{kl}^i\nabla_{\frac{\partial}{\partial x_k}} X_{A_j}\\
       &=&\sum_{k=1}^N\tilde{a}_{kl}^j\sum_{m=1}^N\tilde{a}_{mk}^i\frac{\partial}{\partial x_m} + \sum_{k=1}^N\tilde{a}_{kl}^i\sum_{m=1}^N\tilde{a}_{mk}^j\frac{\partial}{\partial x_m}\\
       &=& \sum_{k=1}^N\sum_{m=1}^N(\tilde{a}_{mk}^i\tilde{a}_{kl}^j + \tilde{a}_{mk}^j\tilde{a}_{kl}^i)\frac{\partial}{\partial x_m}.
    \end{eqnarray*}
\end{proof}

\begin{Lem}\label{l411}
    For $A_i,A_j\in\mathfrak g$, the following are equivalent.
    \begin{description}
\item[(1)]
$$
(d\varphi)_e(A_i)(d\varphi)_e(A_j)+ (d\varphi)_e(A_j)(d\varphi)_e(A_i) = 2\delta_{ij}I_N.
$$
\item[(2)]
$$
\nabla X_{A_i}\circ \nabla X_{A_j} + \nabla X_{A_j}\circ \nabla X_{A_i} = 2\delta_{ij}id_{\mathfrak{X}(\mathbb{R}^N\backslash\{0\})}.
$$
 \end{description}
\end{Lem}
\begin{proof}
Put $\tilde A_i\coloneqq(d\varphi)_e(A_i)$ and
$\tilde A_j\coloneqq(d\varphi)_e(A_j)$. By Lemma~\ref{l410}, the
matrix of
\[
\nabla X_{A_i}\circ\nabla X_{A_j}
+\nabla X_{A_j}\circ\nabla X_{A_i}
\]
with respect to the global frame
$\partial/\partial x_1,\dots,\partial/\partial x_N$ is
$\tilde A_i\tilde A_j+\tilde A_j\tilde A_i$. Therefore, the matrix
identity in (1) is equivalent to the endomorphism identity in (2).
\end{proof}

\begin{Lem}\label{l412}
    Let $n \in \mathbb{N}$. The following are equivalent.
    \begin{description}
\item[(1)]$n \leq \rho^{-}_{G, \mathfrak{s}}(M, \sigma, \nabla)$.
\item[(2)]There exist $A_1, ..., A_n \in \mathfrak{s}$ such that for each $i, j \in \{1, ..., n\}$,
$$
\nabla X_{A_i}\circ \nabla X_{A_j} + \nabla X_{A_j}\circ \nabla X_{A_i} = 2\delta_{ij}id_{\mathfrak{X}(\mathbb{R}^N\backslash\{0\})}.
$$
 \end{description}
\end{Lem}
\begin{proof}
    First, we show (1)$\Rightarrow$(2). Let $f:Cl_{0,n}  \to \textrm{End}_{C^{\infty}(\mathbb{R}^N\backslash\{0\})}(\mathfrak{X}(\mathbb{R}^N\backslash\{0\}))$ be a unital $\mathbb{R}$-algebra homomorphism satisfying $f(e_i) \in \nabla^{\mathbb{R}^N\backslash\{0\}}\mathfrak{s}$ for each $i \in \{1, ..., n\}$. Then for each $i \in \{1, ..., n\}$, there exists $A_i \in \mathfrak{s}$ such that $f(e_i) = \nabla X_{A_i}$. Hence, for each $i, j \in \{1, ..., n\}$,
    \begin{eqnarray*}
        \nabla X_{A_i}\circ \nabla X_{A_j} + \nabla X_{A_j}\circ \nabla X_{A_i}
        &=& f(e_i)f(e_j) +f(e_j)f(e_i)
        = 2\delta_{ij}id_{\mathfrak{X}(\mathbb{R}^N\backslash\{0\})}.
    \end{eqnarray*}

    Next, we show (2)$\Rightarrow$(1). Suppose that there exist $A_1, ..., A_n \in \mathfrak{s}$ such that for each $i, j \in \{1, ..., n\}$,
$$
\nabla X_{A_i}\circ \nabla X_{A_j} + \nabla X_{A_j}\circ \nabla X_{A_i} = 2\delta_{ij}id_{\mathfrak{X}(\mathbb{R}^N\backslash\{0\})}.
$$
Then it suffices to define a unital $\mathbb{R}$-algebra homomorphism
$f:Cl_{0,n}  \to \textrm{End}_{C^{\infty}(\mathbb{R}^N\backslash\{0\})}(\mathfrak{X}(\mathbb{R}^N\backslash\{0\}))$
by $f(e_i) \coloneqq \nabla X_{A_i} \in \nabla^{\mathbb{R}^N\backslash\{0\}}\mathfrak{s}$ for each $i \in \{1, ..., n\}$.
\end{proof}

\begin{proof}[Proof of the second assertion of Theorem~\ref{t13}]
The realification equality above and Theorem~\ref{t44} give
\begin{align*}
\rho^-_{G,\mathfrak p}
(\mathbb C^N\backslash\{0\},\varphi,\nabla)
&=\rho^-_{G,\mathfrak p}
(\mathbb R^{2N}\backslash\{0\},\eta\circ\varphi,F_*\nabla)\\
&=\rho^{(1)}_{\mathfrak p}
\bigl(\mathfrak g,(d\eta)_{I_N}\circ\iota\bigr).
\end{align*}
By Lemma~\ref{lem:realification} and Definition~3.5, the last term is
$\rho^{(1)}(\mathfrak g,\iota)$.
\end{proof}

\subsection{Relation to Clifford structures}

Let $G$ be a Lie group, let $\mathfrak{s}\subset\mathfrak g$ be a vector subspace, let $(M,g)$ be an oriented Riemannian $G$-manifold with group action $\sigma$, and let
$\nabla:\mathfrak X(M)\times\mathfrak X(M)\to\mathfrak X(M)$
be an affine connection on $M$. We write
$\widetilde{\nabla}:\mathfrak X(M)\to \textrm{End}_{C^{\infty}(M)}(\mathfrak{X}(M)) \cong \Gamma(\operatorname{End}(TM))$
for the map induced by $\nabla$.
We define the $\mathbb R$-subspace
\[
\widetilde{\nabla}^{M}\mathfrak s
\coloneqq
\{\widetilde{\nabla}X_A\mid A\in\mathfrak s\}
\subset\Gamma(\operatorname{End}(TM)).
\]

To relate $\rho^+_{G,\mathfrak{s}}(M,\sigma,\nabla)$ to Clifford structures, we introduce the following integers.
\begin{Def}
    For each $\epsilon\in\{\pm1\}$, we define integer
$\rho^\epsilon_{G,\mathfrak{s}}(M, g, \sigma,\widetilde{\nabla})$
as the largest $n\in\mathbb N$ for which there exists a unital $\mathbb R$-algebra homomorphism
$f:Cl_n^\epsilon\to \Gamma(\operatorname{End}(TM))$ 
such that
\begin{itemize}
\item $f(e_i)\in\widetilde{\nabla}^M \mathfrak{s} \cap \Gamma(\textrm{End}^{-}(TM))$ for any $i \in \{1, ..., n\}$.
\end{itemize}
\end{Def}

    It is clear that the following holds:
    $$
    \rho^{\pm}_{G, \mathfrak{s}}(M, g, \sigma, \widetilde{\nabla})\leq \rho^{\pm}_{G, \mathfrak{s}}(M, \sigma, \nabla).
    $$

The following theorem shows that $\rho^+_{G,\mathfrak{s}}(M,g,\sigma,\widetilde{\nabla})$ is closely related to Clifford structures.
\begin{Thm}\label{t414}
    For $n\in\mathbb N$, the following are equivalent.
    \begin{description}
\item[(1)]$n \leq \rho^+_{G, \mathfrak{s}}(M, g, \sigma, \widetilde{\nabla})$.
\item[(2)]There exists a rank $n$ Clifford structure $(E, h, \varphi)$ satisfying the following.
\begin{description}
\item[$\bullet$]$E$ is trivial.
\item[$\bullet$]There exist sections $s_1, ..., s_n \in \Gamma(E)$ forming an $h$-orthonormal frame such that for each $i\in \{1, ..., n\}$, $\varphi\circ s_i \in \widetilde{\nabla}^M\mathfrak{s}$.
 \end{description}
 \end{description}
\end{Thm}
\begin{proof}
    First, we show $(1)\Rightarrow (2)$. Suppose that there exists a unital $\mathbb{R}$-algebra homomorphism $f:Cl_n  \to \Gamma(\textrm{End}(TM))$ such that for each $i \in \{ 1, ..., n \}$, $f(e_i) \in \widetilde{\nabla}^M\mathfrak{s}\cap \Gamma(\textrm{End}^-(TM))$. Let $E \coloneqq M \times \mathbb{R}^n$, and equip it with the standard inner product $h$ and the standard orientation. Then $Cl(E, h) = M \times Cl_n$. We define
    $$
    \varphi:Cl(E, h) \to \textrm{End}(TM), (x, a) \mapsto f(a)(x).
    $$
    This is an algebra bundle homomorphism.
Furthermore, let $s_1, ..., s_n \in \Gamma(E)$ be the sections corresponding to the standard basis $e_1, ... ,e_n$ of $\mathbb{R}^n$. Then $s_1, ..., s_n$ form an $h$-orthonormal frame, and for each $i \in \{1, ..., n\}$, $\varphi\circ s_i = \varphi(-, e_i) = f(e_i) \in \widetilde{\nabla}^M\mathfrak{s}\cap \Gamma(\textrm{End}^-(TM))$. Hence for any $(x, v) \coloneqq (x, \sum_{i=1}^nt_ie_i) \in E \subset Cl(E, h)$,
 $$
 \varphi(x, v) = \sum_{i=1}^nt_i\varphi(x, e_i) =\sum_{i=1}^nt_if(e_i) \in \textrm{End}^-(TM).
 $$
 Therefore, $(E, h, \varphi)$ is a rank $n$ Clifford structure on $(M, g)$.

 Next, we show $(2) \Rightarrow(1)$. Suppose that there exists a rank $n$ Clifford structure $(E, h, \varphi)$ such that:
\begin{description}
\item[$\bullet$]$E$ is trivial.
\item[$\bullet$]There exist sections $s_1, ..., s_n \in \Gamma(E)$ forming an $h$-orthonormal frame such that for each $i\in \{1, ..., n\}$, $\varphi\circ s_i \in \widetilde{\nabla}^M\mathfrak{s}$.
 \end{description}
Then it suffices to define a unital $\mathbb{R}$-algebra homomorphism $f:Cl_n  \to \Gamma(\textrm{End}(TM))$ by $f(e_i) = \varphi\circ s_i \in \widetilde{\nabla}^M\mathfrak{s}$ for each $i \in \{1, ..., n\}$.
\end{proof}

\subsection{Proof of Theorem~\ref{t16}}

Let $\sigma$ be an isometric $G$-action on $M$, and let $\nabla^g$ be the Levi-Civita connection on $M$.
\begin{Lem}
     The following holds:
$$
\widetilde{\nabla^g}^{M}\mathfrak g\subset \Gamma(\operatorname{End}^{-}(TM)).
$$
    \end{Lem}
    \begin{proof}
        If $\sigma$ is isometric, then every fundamental vector field $X \in \mathfrak{X}(M)$ is a Killing vector field. Since $\nabla^g$ is the Levi-Civita connection, for all $Y, Z \in \mathfrak{X}(M)$,
  $$
  g(\nabla^g_YX, Z) + g(Y, \nabla^g_ZX) =0.
  $$
  Therefore, $\widetilde{\nabla^g}^M\mathfrak{g}\subset \Gamma(\textrm{End}^-(TM))$.
    \end{proof}
    
\begin{Cor}\label{c416}
        The following holds:
         $$
    \rho^+_{G, \mathfrak{s}}(M, g, \sigma, \widetilde{\nabla^g})= \rho^{+}_{G, \mathfrak{s}}(M, \sigma, \nabla^g).
    $$
    \end{Cor}

The following follows from Theorem~\ref{t414} and Corollary~\ref{c416}.
\begin{Cor}[Theorem~\ref{t16}]
    The following are equivalent.
    \begin{description}
\item[(1)]$n \leq \rho^+_{G, \mathfrak{s}}(M, \sigma, \nabla^g)$.
\item[(2)]There exists a rank $n$ Clifford structure $(E, h, \varphi)$ satisfying the following.
\begin{description}
\item[$\bullet$]$E$ is trivial.
\item[$\bullet$]There exist sections $s_1, ..., s_n \in \Gamma(E)$ forming an $h$-orthonormal frame such that for each $i\in \{1, ..., n\}$, $\varphi\circ s_i \in \widetilde{\nabla^g}^M\mathfrak{s}$.
 \end{description}
 \end{description}
\end{Cor}

\section*{Acknowledgements}

The author is grateful to Takayuki Okuda, Kazuki Kannaka, and Koichi Tojo for their guidance, encouragement, and work that inspired this paper.
The author is also indebted to Yoshio Agaoka, Akira Kubo, Hiroshi Tamaru, Hiroaki Nagaya, and Ryoya Kai for many helpful comments.
This work was supported by JST SPRING, Grant Number JPMJSP2132.

\addcontentsline{toc}{section}{References}
\phantomsection
\providecommand{\bysame}{\leavevmode\hbox to3em{\hrulefill}\thinspace}
\providecommand{\MR}{\relax\ifhmode\unskip\space\fi MR }
\providecommand{\MRhref}[2]{%
  \href{http://www.ams.org/mathscinet-getitem?mr=#1}{#2}
}

\end{document}